\documentclass{lms4HandOut}
\date{1 December 2011}

\newtheorem{theorem}{Theorem}[section] 
\newtheorem{lemma}[theorem]{Lemma}     

\newtheorem{proposition}[theorem]{Proposition}

\newnumbered{assertion}{Assertion}    
\newnumbered{conjecture}{Conjecture}  
\newnumbered{definition}{Definition}
\newnumbered{hypothesis}{Hypothesis}
\newnumbered{remark}{Remark}
\newnumbered{note}{Note}
\newnumbered{observation}{Observation}
\newnumbered{problem}{Problem}
\newnumbered{question}{Question}
\newnumbered{algorithm}{Algorithm}
\newnumbered{example}{Example}
\newunnumbered{notation}{Notation} 

\usepackage{amsmath,amssymb,latexsym}

\simpleequations

\newcommand{\Kb}[1]{\ensuremath{\left\langle{#1}\right\rangle}}
\newcommand{\ft}[3][0]{\ensuremath{{#1}\le{#2}\le{#3}}}

\newcommand{\putln}[5]{\put(#1,#2){\line(#3,#4){#5}}}
\newcounter{x}
\newcounter{y}
\newcommand{\setax}[2]{\setcounter{x}{#1}\setcounter{y}{#2}}
\newcommand{\addxy}[2]{%
\addtocounter{x}{#1}%
\addtocounter{y}{#2}%
}
\newcommand{\lnfrom}[3]{\putln{\value{x}}{\value{y}}{#1}{#2}{#3}}
\newcommand{\putxy}[1]{\put(\value{x},\value{y}){#1}}

\title{On Jones polynomials of alternating pretzel knots}
\author{Masao Hara and Makoto Yamamoto}
\classno{57M27}
\dedication{Dedicated to Professor Sin'ichi Suzuki on the occasion of
his 70th birthday}

\begin{document}
\maketitle
\begin{abstract}
 We show that there are infinitely many pairs of alternating
 pretzel knots whose Jones polynomials are identical.
\end{abstract}

\section{Introduction}
After Jones \cite{Jones1987} defined a polynomial invariant of links, 
new polynomial invariants of links were defined,
namely Homfly polynomial \cite{HOMFLY1985} by
Freyd-Yetter-Hoste-Lickorish-Millett, 
Kauffman Polynomial \cite{Kauffman1990} by Kauffman 
and $Q$-polynomial \cite{BLM1986} by Brandt-Lickorish-Millett.
Kauffman \cite{Kauffman1987} defined bracket polynomials of link
diagrams and  showed the relation between bracket polynomial and Jones
polynomial.
They have been used to estimate crossing numbers of links 
\cite{HTY2000}, \cite{HY1992}, \cite{LT1988}, \cite{Murasugi1987},
settle Tait conjectures
\cite{Murasugi1987}, \cite{Murasugi1988Cambridge},
\cite{Thistlethwaite1987} 
and inspect symmetries of knots 
\cite{Murasugi1988Pacific}, \cite{Traczyk1990}, 
\cite{Yokota1991AMS}, \cite{Yokota1991Annalen}.

On the other hand, 
it seems that Jones polynomial is not well understood in term of intrinsic
topological properties of links.
It is unknown whether there exists a non-trivial knot whose Jones
polynomial is equal to $1$.
In case of links, Eliahou-Kauffman-Thistlethwaite \cite{EKT2003} exhibited
$k$-component nontrivial links $(k\ge2)$ whose Jones polynomials are equal to
that of the $k$-component trivial link.

In this paper, we exhibit infinitely many pairs of alternating pretzel knots
whose Jones polynomials are identical but Alexander polynomials are
distinct.
The second author and his students, Asami, Mori and Senuma, 
made programs that calculate
bracket polynomials of pretzel links, 
and made a table of bracket polynomials of alternating pretzel knots
such that their crossing numbers are less than or equal to 100.
There are $28\,289\,375$ alternating pretzel knots with their crossing
numbers are less than or equal to  100, 
there are  twelve pairs of them satisfying with above property.

For a link diagram $D$, we denote the set of crossings of $D$ by
$\mathcal{C}(D)$ and the crossing number of $D$ by $c(D)$.
A function $s:\mathcal{C}(D)\maketitle\longrightarrow\{-1,+1\}$
is called a \emph{state} of $D$, and the set of states of $D$
is denoted by $\mathcal{S}(D)$. 
For a state $s$ of $D$, we set
\[
 \Kb{D|s} = A^{\sum_{c\in \mathcal{C}(D)}s(c)}
\left(-A^2-A^{-2}\right)^{|sD|-1},
\]
where $sD$ is $D$ with all of its crossings nullified according to
the rule depicted by Figure~\ref{figure:nullify},
and $|sD|$ is the number of connected components of $sD$.
Kauffman bracket polynomial $\displaystyle\Kb{D}$ of $D$ is defined by
\[
 \Kb D = \sum_{s\in\mathcal{S}(D)}\Kb{D|s}.
\] 
For an oriented link $L$ and its diagram $D$,
we have
\[
 V_L(A^{-4}) = (-A)^{-3w(D)}\Kb D,
\]
where $V_L$ is the Jones polynomial of $L$ and $w(D)$ the writhe of $D$.
This definition of Jones polynomial is due to Kauffman \cite{Kauffman1987}.

\begin{figure}
\begin{center}
  \begin{picture}(210,60)
   \thicklines
   \put(0,60){\line(1,-1){20}}
   \put(50,10){\line(-1,1){20}}
   \put(0,10){\line(1,1){50}}
   \put(25,9){\makebox(0,0)[t]{$c$}}
   \qbezier(80,10)(120,35)(80,60)
   \qbezier(130,10)(90,35)(130,60)
   \put(105,9){\makebox(0,0)[t]{$s(c)=1$}}
   \qbezier(160,10)(185,50)(210,10)
   \qbezier(160,60)(185,20)(210,60)
   \put(185,9){\makebox(0,0)[t]{$s(c)=-1$}}
  \end{picture}\\
\caption{\label{figure:nullify}}
\end{center}
\end{figure}

Let $D(p_1,\dots,p_n)$ be the pretzel diagram shown
in Figure~\ref{figure:pretzel},
and $P(p_1,\dots,p_n)$ the link represented by $D(p_1,\dots,p_n)$.
We denote 
\[
D(\underbrace{1,\dots,1}_m,p_1,\dots,p_n) \mbox{\quad and\quad }
P(\underbrace{1,\dots,1}_m,p_1,\dots,p_n)
\]
by $D(m;p_1,\dots,p_n)$ and $P(m;p_1,\dots,p_n)$ respectively for short.
Particularly $D(0;p_1,\dots,p_n)=D(p_1,\dots,p_n)$ and
$P(0;p_1,\dots,p_n)=P(p_1,\dots,p_n)$.

\begin{figure}
\begin{center}
 \begin{picture}(160,140)(0,0)
  \thicklines
  \multiput(0,0)(160,0){2}{\line(0,1){10}}
  \multiput(0,130)(160,0){2}{\line(0,1){10}}
  \multiput(0,0)(0,140){2}{\line(1,0){160}}
  \multiput(30,10)(0,120){2}{\line(1,0){30}}
  \multiput(90,10)(0,120){2}{\line(1,0){10}}
  \multiput(120,10)(0,120){2}{\line(1,0){10}}
  \multiput(100,10)(4,0){5}{\line(1,0){2}}
  \multiput(100,130)(4,0){5}{\line(1,0){2}}
  \setax{0}{10} 
  \lnfrom{1}{1}{30} \addxy{30}0 \lnfrom{-1}{1}{12} 
  \addxy{-30}{30} \lnfrom{1}{-1}{12}
  \addxy01 
  \multiput(\value{x},\value{y})(0,4){7}{\line(0,1){2}}
  \addxy{30}{0} \multiput(\value{x},\value{y})(0,4){7}{\line(0,1){2}}
  \addxy{-30}{29} 
  \lnfrom{1}{1}{30} \addxy{30}0 \lnfrom{-1}{1}{12} 
  \addxy{-30}{30} \lnfrom{1}{-1}{12}
  \lnfrom{1}{1}{30} \addxy{30}0 \lnfrom{-1}{1}{12} 
  \addxy{-30}{30} \lnfrom{1}{-1}{12}
  \addxy{-1}{-45}
  \putxy{\makebox(0,0)[r]{$\displaystyle p_1$}}
  \setax{60}{10} 
  \lnfrom{1}{1}{30} \addxy{30}0 \lnfrom{-1}{1}{12} 
  \addxy{-30}{30} \lnfrom{1}{-1}{12}
  \addxy01 
  \multiput(\value{x},\value{y})(0,4){7}{\line(0,1){2}}
  \addxy{30}{0} \multiput(\value{x},\value{y})(0,4){7}{\line(0,1){2}}
  \addxy{-30}{29} 
  \lnfrom{1}{1}{30} \addxy{30}0 \lnfrom{-1}{1}{12} 
  \addxy{-30}{30} \lnfrom{1}{-1}{12}
  \lnfrom{1}{1}{30} \addxy{30}0 \lnfrom{-1}{1}{12} 
  \addxy{-30}{30} \lnfrom{1}{-1}{12}
  \addxy{-1}{-45} \putxy{\makebox(0,0)[r]{$\displaystyle p_2$}}
  \setax{130}{10} 
  \lnfrom{1}{1}{30} \addxy{30}0 \lnfrom{-1}{1}{12} 
  \addxy{-30}{30} \lnfrom{1}{-1}{12}
  \addxy01 
  \multiput(\value{x},\value{y})(0,4){7}{\line(0,1){2}}
  \addxy{30}{0} \multiput(\value{x},\value{y})(0,4){7}{\line(0,1){2}}
  \addxy{-30}{29} 
  \lnfrom{1}{1}{30} \addxy{30}0 \lnfrom{-1}{1}{12} 
  \addxy{-30}{30} \lnfrom{1}{-1}{12}
  \lnfrom{1}{1}{30} \addxy{30}0 \lnfrom{-1}{1}{12} 
  \addxy{-30}{30} \lnfrom{1}{-1}{12}
  \addxy{-1}{-45} \putxy{\makebox(0,0)[r]{$\displaystyle p_n$}}
 \end{picture}\\
\caption{\label{figure:pretzel}}
\end{center} 
\end{figure} 

\begin{theorem} \label{theorem:bracket}
 Let $a,b,c$ be positive integers and $m$ a non-negative even number.
\begin{eqnarray*}
\lefteqn{ \Kb{D(m;a,b,c)}}\\
&=&A^{3m+a+b+c}\delta^{-4}\left(
\left(F_a+\delta^2\right)\left(F_b+\delta^2\right)\left(F_c+\delta^2\right)
+(-A^{-4})^m(\delta^2-1)F_aF_bF_c
\right),
\end{eqnarray*}
where $\delta=-A^2-A^{-2}$ and
$F_i=(-A^{-4})^i-1$.
\end{theorem}

\begin{theorem}\label{theorem:Jones}
 For a non-negative even number $k$, 
\[
V_{P(k;k+4,k+3,k+5)}=V_{P(k+6;k+2,k+1,k+3)}. 
\]
\end{theorem}

\begin{proposition}\label{prop:Alexander}
 For a non-negative even number $k$, 
\[
\Delta_{P(k;k+4,k+3,k+5)}\neq\Delta_{P(k+6;k+2,k+1,k+3)}. 
\]
\end{proposition}

We note that Homfly polynomials of $P(k;k+4,k+3,k+5)$ and $P(k+6;k+2,k+1,k+3)$ are distinct, 
since Alexander polynomial of knots is the specific value of Homfly polynomial
of them.  

\section{Proof of Theorem~\ref{theorem:bracket}}

\begin{figure}
\begin{center}
 \begin{picture}(310,180)(-10,0)
  \thicklines
  \putln0{105}01{75} \putln0{75}0{-1}{75}
  \putln{120}{105}01{45} \putln{120}{75}0{-1}{45}
  \multiput(120,30)(60,0){3}{\line(1,0){30}}
  \multiput(120,150)(60,0){3}{\line(1,0){30}}
  \multiput(0,0)(0,180){2}{\line(1,0){300}}
  \multiput(300,0)(0,150){2}{\line(0,1){30}}
  \setax{0}{75} 
  \lnfrom{1}{1}{30} \addxy{0}{30} \lnfrom{1}{-1}{12} 
  \addxy{30}{-30} \lnfrom{-1}{1}{12}
  \lnfrom{1}{1}{30} \addxy{0}{30} \lnfrom{1}{-1}{12} 
  \addxy{30}{-30} \lnfrom{-1}{1}{12}
  \addxy10 
  \multiput(\value{x},\value{y})(4,0){7}{\line(1,0){2}}
  \addxy{0}{30} \multiput(\value{x},\value{y})(4,0){7}{\line(1,0){2}}
  \addxy{29}{-30} 
  \lnfrom{1}{1}{30} \addxy{0}{30} \lnfrom{1}{-1}{12} 
  \addxy{30}{-30} \lnfrom{-1}{1}{12}
  \addxy{-75}{-1}
  \putxy{\makebox(0,0)[t]{$m$}}
  \setax{150}{30} 
  \lnfrom{1}{1}{30} \addxy{30}0 \lnfrom{-1}{1}{12} 
  \addxy{-30}{30} \lnfrom{1}{-1}{12}
  \addxy01 
  \multiput(\value{x},\value{y})(0,4){7}{\line(0,1){2}}
  \addxy{30}{0} \multiput(\value{x},\value{y})(0,4){7}{\line(0,1){2}}
  \addxy{-30}{29} 
  \lnfrom{1}{1}{30} \addxy{30}0 \lnfrom{-1}{1}{12} 
  \addxy{-30}{30} \lnfrom{1}{-1}{12}
  \lnfrom{1}{1}{30} \addxy{30}0 \lnfrom{-1}{1}{12} 
  \addxy{-30}{30} \lnfrom{1}{-1}{12}
  \addxy{-1}{-45}
  \putxy{\makebox(0,0)[r]{$a$}}
  \setax{210}{30} 
  \lnfrom{1}{1}{30} \addxy{30}0 \lnfrom{-1}{1}{12} 
  \addxy{-30}{30} \lnfrom{1}{-1}{12}
  \addxy01 
  \multiput(\value{x},\value{y})(0,4){7}{\line(0,1){2}}
  \addxy{30}{0} \multiput(\value{x},\value{y})(0,4){7}{\line(0,1){2}}
  \addxy{-30}{29} 
  \lnfrom{1}{1}{30} \addxy{30}0 \lnfrom{-1}{1}{12} 
  \addxy{-30}{30} \lnfrom{1}{-1}{12}
  \lnfrom{1}{1}{30} \addxy{30}0 \lnfrom{-1}{1}{12} 
  \addxy{-30}{30} \lnfrom{1}{-1}{12}
  \addxy{-1}{-45} \putxy{\makebox(0,0)[r]{$\displaystyle b$}}
  \setax{270}{30} 
  \lnfrom{1}{1}{30} \addxy{30}0 \lnfrom{-1}{1}{12} 
  \addxy{-30}{30} \lnfrom{1}{-1}{12}
  \addxy01 
  \multiput(\value{x},\value{y})(0,4){7}{\line(0,1){2}}
  \addxy{30}{0} \multiput(\value{x},\value{y})(0,4){7}{\line(0,1){2}}
  \addxy{-30}{29} 
  \lnfrom{1}{1}{30} \addxy{30}0 \lnfrom{-1}{1}{12} 
  \addxy{-30}{30} \lnfrom{1}{-1}{12}
  \lnfrom{1}{1}{30} \addxy{30}0 \lnfrom{-1}{1}{12} 
  \addxy{-30}{30} \lnfrom{1}{-1}{12}
  \addxy{-1}{-45} \putxy{\makebox(0,0)[r]{$\displaystyle c$}}
  \thinlines
  \put(60,90){\oval(130,50)}
  \put(60,120){\makebox(0,0)[b]{$\displaystyle \mathcal{C}_0$}}
  \put(165,90){\oval(50,130)}
  \put(165,160){\makebox(0,0)[b]{$\displaystyle \mathcal{C}_1$}}
  \put(225,90){\oval(50,130)}
  \put(225,160){\makebox(0,0)[b]{$\displaystyle \mathcal{C}_2$}}
  \put(285,90){\oval(50,130)}
  \put(285,160){\makebox(0,0)[b]{$\displaystyle \mathcal{C}_3$}}
 \end{picture}
\caption{\label{figure:Dmabc}}
\end{center} 
\end{figure} 

Let $D$ be $D(m;a,b,c)$ and 
$\displaystyle\mathcal{C}_i\,(i=0,1,2,3)$ the subset of
$\mathcal C(D)$  as shown by Figure~\ref{figure:Dmabc}.
By $\displaystyle \mathcal S_{i,j,k,l}$,
we denote the set of states $s$ of $D$
such that
$\left|s_1^{-1}(-1)\cap\mathcal{C}_0\right|=i$,
$\left|s_1^{-1}(-1)\cap\mathcal{C}_1\right|=j$,
$\left|s_1^{-1}(-1)\cap\mathcal{C}_2\right|=k$ and
$\left|s_1^{-1}(-1)\cap\mathcal{C}_3\right|=l$
for non-negative integers $i,j,k,l$.
Obviously, if $ s_1,s_2\in\mathcal S_{i,j,k,l}$, 
then $  \sum_{c\in\mathcal C(D)}s_1(c) 
= \sum_{c\in\mathcal C(D)}s_2(c)$,
and $\displaystyle \left|s_1D\right| = \left|s_2D\right|$.
For $\displaystyle s\in\mathcal S_{i,j,k,l}$, we denote 
$ \sum_{c\in\mathcal  C(D)} s(c)$ and
$ |sD|$ by $\sigma(i,j,k,l)$ and $\mu(i,j,k,l)$
respectively.
Then 
\[
 \Kb D = \sum_{s\in\mathcal S(D)}\Kb{D|s}
 = \sum_{i,j,k,l}\sum_{s\in\mathcal S(i,j,k,l)}
A^{\sigma(i,j,k,l)}\delta^{\mu(i,j,k,l)-1}.
\]
Since the number of state in $S_{i,j,k,l}$ is 
$\binom mi\binom{a}{j}\binom{b}{k}\binom{c}{l}$ and
$\sigma(i,j,k,l)=m+a+b+c-2(i+j+k+l)$ ,
we have
\begin{equation}
 \label{equation:BrackeSum}
 \Kb D = 
\sum_{(i,j,k,l)\in X}
 \binom mi\binom{a}{j}\binom{b}{k}\binom{c}{l}
 A^{m+a+b+c-2(i+j+k+l)}\delta^{\mu(i,j,k,l)-1},
\end{equation}
where 
\[
X=\{(i,j,k,l)\,|\,0\le i\le m, 0\le j\le a,
0\le k\le b, 0\le l\le c\}.
\]

We decompose $X$ into $4$ mutually disjoint subsets
$X_0, X_1, X_2$ and $X_3$,
where
\begin{eqnarray*}
 X_0 &=& \{(i,j,k,l)\in X\,|\,j=k=l=0\},\\
 X_1 &=& \{(i,j,k,l)\in X\setminus X_0\,|\,j=k=0, k=l=0\,\mbox{or}\,l=j=0\},\\
 X_2 &=& \{(i,j,k,l)\in X\setminus(X_0\cup X_1)\,|\,
  j=0, l=0 \,\mbox{or}\, k=0\},\\
 X_3&=& X \setminus \left(X_0\cup X_1\cup X_2\right).
\end{eqnarray*}
This means that $X=X_0{\cup}X_1{\cup}X_2{\cup}X_3$ and
$X_{\alpha}{\cap}X_{\beta}=\emptyset,\,(0\le\alpha<\beta\le3).$

\begin{lemma}\label{lemma:X0}
\[
 \sum_{(i,j,k,l)\in X_0}  
 \binom mi\binom{a}{j}\binom{b}{k}\binom{c}{l}
A^{m+a+b+c-2(i+j+k+l)}\delta^{\mu(i,j,k,l)-1}=A^{3m+a+b+c}\delta^2.
\]
\end{lemma}
\begin{proof*}
 If $(i,j,k,l)\in X_0$, then $j=k=l=0$ and $\mu(i,0,0,0)=m-i+3$. 
 It follows that
 \begin{eqnarray*} 
 \lefteqn{\sum_{(i,j,k,l)\in X_0}   
 \binom mi\binom{a}{j}\binom{b}{k}\binom{c}{l} 
 A^{m+a+b+c-2(i+j+k+l)}\delta^{\mu(i,j,k,l)-1}}\\ 
 &=&\sum_{i=0}^{m}\binom mi A^{m+a+b+c-2i}\delta^{m-i+3-1}
 =A^{m+a+b+c}\delta^2\sum_{i=0}^{m}\binom mi(A^{-2})^i\delta^{m-i}\\
 &=&A^{m+a+b+c}\delta^2(A^{-2}+\delta)^m
 \end{eqnarray*}
 Since $m$ is an even number and $A^2+\delta=-A^2$, we have
\[
 \sum_{(i,j,k,l)\in X_0}  
 \binom mi\binom{a}{j}\binom{b}{k}\binom{c}{l}
 A^{m+a+b+c-2(i+j+k+l)}\delta^{\mu(i,j,k,l)-1}
 =A^{3m+a+b+c}\delta^2. \makebox[0cm]{\hspace{22mm}\proofbox}
\] 
\end{proof*}

\begin{lemma}\label{lemma:X1}
\[
 \sum_{(i,j,k,l)\in X_1}  
 \binom mi\binom{a}{j}\binom{b}{k}\binom{c}{l}
A^{m+a+b+c-2(i+j+k+l)}\delta^{\mu(i,j,k,l)-1}=
A^{3m+a+b+c} (F_a+F_b+F_c).
\]
\end{lemma}
\begin{proof}
 We decompose $X_1$ into 3 mutually disjoint subsets
 $X_{11}, X_{12}$ and $X_{13}$, where
 \begin{eqnarray*}
 X_{11} &=& \{(i,j,k,l)\in X_1\,|\,j>0\},\\
 X_{12} &=& \{(i,j,k,l)\in X_1\,|\,k>0\},\\
 X_{13} &=& \{(i,j,k,l)\in X_1\,|\,l>0\}.
 \end{eqnarray*}
 If $(i,j,k,l)\in X_{11}$, then $k=l=0$, $j>0$ and $\mu(i,j,0,0)=m-i+j+1$. 
 It follows that
 \begin{eqnarray*} 
 \lefteqn{\sum_{(i,j,k,l)\in X_{11}}   
 \binom mi\binom{a}{j}\binom{b}{k}\binom{c}{l} 
 A^{m+a+b+c-2(i+j+k+l)}\delta^{\mu(i,j,k,l)-1}}\\ 
 &=&\sum_{i=0}^{m}\sum_{j=1}^a \binom mi\binom aj 
 A^{m+a+b+c-2(i+j)}\delta^{m-i+j}\\
 &=&\sum_{j=1}^a\left(\binom aj A^{m+a+b+c-2j}\delta^{j}
 \sum_{i=0}^m\binom mi A^{-2i}\delta^{m-i}\right)\\
 &=&\sum_{j=1}^a\left(\binom aj A^{m+a+b+c-2j}\delta^{j}
 A^{2m}\right)
 =A^{3m-a+b+c}\sum_{j=1}^a\left(\binom aj
 A^{2(a-j)}\delta^j\right)\\
 &=&A^{3m-a+b+c}\delta\left((A^2+\delta)^a-A^{2a}\right)
 =A^{3m-a+b+c}\left((-A^{-2})^a-A^{2a}\right)\\
 &=&A^{3m+a+b+c}\left((-A^{-4})^a-1\right)
 =A^{3m+a+b+c}F_a.
 \end{eqnarray*}
 Similarly, we get
 \begin{eqnarray*}
 \sum_{(i,j,k,l)\in X_{12}}   
 \binom mi\binom{a}{j}\binom{b}{k}\binom{c}{l} 
 A^{m+a+b+c-2(i+j+k+l)}\delta^{\mu(i,j,k,l)-1}
 &=&A^{3m+a+b+c} F_b,\\
 \sum_{(i,j,k,l)\in X_{13}}   
 \binom mi\binom{a}{j}\binom{b}{k}\binom{c}{l} 
 A^{m+a+b+c-2(i+j+k+l)}\delta^{\mu(i,j,k,l)-1}
 &=&A^{3m+a+b+c} F_c.
 \end{eqnarray*}
These imply that Lemma~\ref{lemma:X1}.
\end{proof}

\begin{lemma}\label{lemma:X2}
\begin{eqnarray*}
 \lefteqn{\sum_{(i,j,k,l)\in X_2}  
 \binom mi\binom{a}{j}\binom{b}{k}\binom{c}{l}
A^{m+a+b+c-2(i+j+k+l)}\delta^{\mu(i,j,k,l)-1}}\\
&=&A^{3m+a+b+c}\delta^{-2} (F_aF_b+F_bF_c+F_cF_a).
\end{eqnarray*}
\end{lemma}
\begin{proof}
 We decompose $X_2$ into 3 mutually disjoint subsets
 $X_{21}, X_{22}$ and $X_{23}$, where
 \begin{eqnarray*}
 X_{21} &=& \{(i,j,k,l)\in X_2\,|\,l=0\},\\
 X_{22} &=& \{(i,j,k,l)\in X_2\,|\,j=0\},\\
 X_{23} &=& \{(i,j,k,l)\in X_2\,|\,k=0\}.
 \end{eqnarray*}
 If $(i,j,k,l)\in X_{21}$, then $j,k>0$, $l=0$ and $\mu(i,j,k,0)=m-i+j+k-1$. 
 It follows that
 \begin{eqnarray*} 
 \lefteqn{\sum_{(i,j,k,l)\in X_{21}}   
 \binom mi\binom{a}{j}\binom{b}{k}\binom{c}{l} 
 A^{m+a+b+c-2(i+j+k+l)}\delta^{\mu(i,j,k,l)-1}}\\ 
 &=&\sum_{i=0}^{m}\sum_{j=1}^a\sum_{k=1}^b \binom mi\binom aj\binom bk 
 A^{m+a+b+c-2(i+j+k)}\delta^{m-i+j+k-2}\\
 &=&\sum_{\ft[1]ja,\ft[1]kb}\left(
 \binom aj\binom bk A^{m+a+b+c-2j-2k}\delta^{j+k-2}
 \sum_{i=0}^m\binom mi A^{-2i}\delta^{m-i}\right)\\
 &=&A^{3m-a-b+c}\delta^{-2}\sum_{\ft[1]ja,\ft[1]kb}
 \left(\binom aj A^{2(a-j)}\delta^{j}\binom bk A^{2(b-k)}\delta^{k}\right)\\
 &=&A^{3m-a-b+c}\delta^{-2}\left(\left(A^2+\delta\right)^a-A^{2a}\right)
 \left(\left(A^2+\delta\right)^b-A^{2b}\right)\\
 &=& A^{3m+a+b+c}\delta^{-2} F_a F_b.
 \end{eqnarray*}
 It follows Lemma~\ref{lemma:X2}.
\end{proof}

\begin{lemma}\label{lemma:X3}
\begin{eqnarray*}
 \lefteqn{\sum_{(i,j,k,l)\in X_3}  
 \binom mi\binom{a}{j}\binom{b}{k}\binom{c}{l}
A^{m+a+b+c-2(i+j+k+l)}\delta^{\mu(i,j,k,l)-1}}\\
&=&A^{3m+a+b+c}\delta^{-4} 
\left(F_aF_bF_c+(\delta^2-1)A^{-4m}F_aF_bF_c\right).
\end{eqnarray*}
\end{lemma}
\begin{proof}
 We decompose $X_3$ into 2 disjoint subsets
 $X_{31}$ and $X_{32}$, where
 $\displaystyle X_{31} = \{(i,j,k,l)\in X_3\,|\,i<m\}$ and
 $\displaystyle X_{32} = \{(i,j,k,l)\in X_3\,|\,i=m\}$.
 We have
 \[
 \mu(i,j,k,l) = \left\{
 \begin{array}{ll}
 m-i+j+k+l-3,\quad & (i,j,k,l)\in X_{31}, \\
 j+k+l-1, & (i,j,k,l)\in X_{32}.
 \end{array}
 \right.
 \]
 It follows that
 \begin{eqnarray*}
 \lefteqn{\sum_{(i,j,k,l)\in X_{31}}  
 \binom mi\binom{a}{j}\binom{b}{k}\binom{c}{l}
 A^{m+a+b+c-2(i+j+k+l)}\delta^{\mu(i,j,k,l)-1}}\\ 
 &=& \sum_{i=0}^{m-1}\sum_{j=1}^{a}\sum_{k=1}^{b}\sum_{l=1}^{c}
 \binom mi\binom{a}{j}\binom{b}{k}\binom{c}{l} 
 A^{m+a+b+c-2(i+j+k+l)}\delta^{m-i+j+k+l-4}\\
 &=& A^{m-a-b-c} \delta^{-4}
 \left(\sum_{i=0}^{m-1}\binom miA^{-2i}\delta^{m-i} \right)
 \left(\sum_{j=1}^{a}\binom aj A^{2(a-j)}\delta^{j} \right)
 \left(\sum_{k=1}^{b}\binom bk A^{2(b-k)}\delta^{k} \right)\\
 &&\times\left(\sum_{l=1}^{c}\binom cl A^{2(c-l)}\delta^{l} \right)\\
 &=&A^{m+a+b+c}\delta^{-4}\left(A^{2m}-A^{-2m}\right)F_aF_bF_c
 =A^{3m+a+b+c}\delta^{-4}\left(F_aF_bF_c-A^{-4m}F_aF_bF_c\right).
 \end{eqnarray*}
 We have
 \begin{eqnarray*}
 \lefteqn{\sum_{(i,j,k,l)\in X_{32}}  
 \binom mi\binom{a}{j}\binom{b}{k}\binom{c}{l}
 A^{m+a+b+c-2(i+j+k+l)}\delta^{\mu(i,j,k,l)-1}}\\ 
 &=& \sum_{j=1}^{a}\sum_{k=1}^{b}\sum_{l=1}^{c}
 \binom{a}{j}\binom{b}{k}\binom{c}{l} 
 A^{m+a+b+c-2(m+j+k+l)}\delta^{j+k+l-2}\\
 &=& A^{-m-a-b-c} \delta^{-2}
 \left(\sum_{j=1}^{a}\binom aj A^{2(a-j)}\delta^{j} \right)
 \left(\sum_{k=1}^{b}\binom bk A^{2(b-k)}\delta^{k} \right)
 \left(\sum_{l=1}^{c}\binom cl A^{2(c-l)}\delta^{l} \right)\\
 &=&A^{-m+a+b+c}\delta^{-2}\left(A^{2m}-A^{-2m}\right)F_aF_bF_c
 =A^{3m+a+b+c}\delta^{-4}\left(A^{-4m}\delta^2F_aF_bF_c\right). 
 \end{eqnarray*}
 The property that $m$ is an even number implies
 Lemma~\ref{lemma:X3}.
\end{proof}

\noindent
\begin{proof*}[of Theorem~\ref{theorem:bracket}]
 By Lemmas~\ref{lemma:X0}, \ref{lemma:X1}, \ref{lemma:X2} and \ref{lemma:X3},  
 Equation~(\ref{equation:BrackeSum}) implies that
 \begin{eqnarray*}
  \Kb D &=& A^{3m+a+b+c}\delta^2+A^{3m+a+b+c} (F_a+F_b+F_c)\\
 &&+A^{3m+a+b+c}\delta^{-2} (F_aF_b+F_bF_c+F_cF_a)\\
 &&+A^{3m+a+b+c}\delta^{-4} 
 \left(F_aF_bF_c+(\delta^2-1)A^{-4m}F_aF_bF_c\right)\\
 &=& A^{3m+a+b+c}\delta^{-4}\left(
 F_aF_bF_c+\delta^2(F_aF_b+F_bF_c+F_cF_a)+\delta^4(F_a+F_b+F_c)+\delta^6\right.\\
 &&+\left.(\delta^2-1)A^{-4m}F_aF_bF_c
 \right).
 \end{eqnarray*}
 Since $m$ is an even number, we have
 \[
  \Kb D = A^{3m+a+b+c}\delta^{-4}\left(
 \left(F_a+\delta\right)\left(F_b+\delta\right)\left(F_c+\delta\right)
 +(-A^{-4})^m\left(\delta^2-1\right)F_aF_bF_c
 \right).\makebox[0cm]{\hspace{21mm}\proofbox}
 \]
\end{proof*}

\section{Proof of Theorem~\ref{theorem:Jones}}
\label{161319_29Sep11}
For $\delta=-A^2-A^{-2}$, we have 
\begin{equation}\label{euqation:delta}
 A^{-4}(1-A^{-4})(\delta^2-1)=1+(-A^{-4})^3.
\end{equation}
For $F_a=(-A^{-4})^a-1$, we get
\begin{equation}\label{equation:F}
 F_{a+b} = (-A^{-4})^b F_a + F_b, \quad F_aF_b=F_{a+b}-F_a-F_b, 
\end{equation}
and $\displaystyle  F_6 = A^{-4}F_3(1-A^{-4})(\delta^2-1)$.

\begin{lemma}\label{lemma:TheoremJones1}
\[
 \left(F_{a+2}+\delta^2\right)\left(F_{a+1}+\delta^2\right)-
 \left(F_{a-1}+\delta^2\right)\left(F_{a-2}+\delta^2\right)=
 (-A^{-4})^{a-3}F_6F_{a}.
\] 
\end{lemma}
\begin{proof*}
 Proof.~By (\ref{equation:F}), it follows that
 \begin{eqnarray*}
 \lefteqn{ \left(F_{a+2}+\delta^2\right)\left(F_{a+1}+\delta^2\right)-
 \left(F_{a-1}+\delta^2\right)\left(F_{a-2}+\delta^2\right)}\\
 &=& F_{a+2}F_{a+1}-F_{a-1}F_{a-2}
 +\delta^2\left(F_{a+2}+F_{a+1}-F_{a-1}-F_{a-2}\right)\\
 &=& F_{2a+3}-F_{2a-3}+
 (\delta^2-1)\left(F_{a+2}+F_{a+1}-F_{a-1}-F_{a-2}\right)\\
 &=& (-A^{-4})^{2a-3}F_6
 +(\delta^2-1)\left((-A^{-4})^{a-1}F_3+(-A^{-4})^{a-2}F_3\right)\\
 &=& (-A^{-4})^{2a-3}F_6+(-A^{-4})^{a-2}F_3(\delta^2-1)(1-A^{-4})\\
 &=&  F_6\left((-A^{-4})^{2a-3}-(-A^{-4})^{a-3}\right)
 =(-A^{-4})^{a-3}F_6F_{a}.\makebox[0cm]{\hspace*{70mm}\proofbox}
 \end{eqnarray*}
\end{proof*}

\begin{lemma}\label{lemma:TheoremJones2}
\[
 (\delta^2-1)\left(F_{a+2}F_{a+1}-(-A^{-4})^6F_{a-1}F_{a-2}\right)=
 -F_6\left(F_{a}+\delta^2\right).
\] 
\end{lemma}
\begin{proof*}
 We have
 \begin{eqnarray*}
 \lefteqn{F_{a+2}F_{a+1}-(-A^{-4})^6F_{a-1}F_{a-2}}\\
 &=&\left((-A^{-4})^{a+2}-1)\right)\left((-A^{-4})^{a+1}-1)\right)
 -(-A^{-4})^6\left((-A^{-4})^{a-1}-1)\right)
 \left((-A^{-4})^{a-2}-1)\right)\\
 &=&(-A^{-4})^{a+5}+(-A^{-4})^{a+4}-(-A^{-4})^{a+2}-(-A^{-4})^{a+1}
 -(-A^{-4})^{6}+1\\
 &=& (-A^{-4})^{a+1}(1-A^{-4})F_3-F_6.
 \end{eqnarray*}
 It implies that
 \begin{eqnarray*}
 \lefteqn{
  (\delta^2-1)\left(F_{a+2}F_{a+1}-(-A^{-4})^6F_{a-1}F_{a-2}\right)}\\
 &=& (\delta^2-1)(-A^{-4})^{a+1}(1-A^{-4})F_3-(\delta^2-1)F_6\\
 &=& -(-A^{-4})^{a}F_6-(\delta^2-1)F_6
 =-\left((-A^{-4})^{a}-1+\delta^2\right)F_6\\
 &=&-F_6(F_{a}+\delta^2).\makebox[0cm]{\hspace{177mm}\proofbox}
 \end{eqnarray*}
\end{proof*}

\begin{proof}[of Theorem~\ref{theorem:Jones}]
 By Lemma~\ref{lemma:TheoremJones1}, we have
 \begin{eqnarray*}
 \lefteqn{%
 \left(F_{k+4}+\delta^2\right)%
 \left(F_{k+3}+\delta^2\right)%
 \left(F_{k+5}+\delta^2\right)}\\%
 \lefteqn{-\left(F_{k+2}+\delta^2\right)%
 \left(F_{k+1}+\delta^2\right)%
 \left(F_{k+3}+\delta^2\right)
 }\\
 &=& (-A^{-4})^kF_6F_{k+3}\left(F_{k+3}+\delta^2\right).
 \end{eqnarray*}
 By Lemma~\ref{lemma:TheoremJones2}, we also have
 \begin{eqnarray*}
 \lefteqn{(-A^{-4})^kF_{k+4}F_{k+3}F_{k+5}%
 -(-A^{-4})^{k+6}F_{k+2}F_{k+1}F_{k+3}}\\
 &=& -(-A^{-4})^kF_6F_{k+3}\left(F_{k+3}+\delta^2\right)
 \end{eqnarray*}
 These imply that
 \begin{eqnarray*}\nonumber
 \lefteqn{\left(F_{k+4}+\delta^2\right) 
 \left(F_{k+3}+\delta^2\right) \left(F_{k+5}+\delta^2\right)
 +(-A^{-4})^{k}(\delta^2-1)F_{k+4}F_{k+3}F_{k+5}}\\
 &=& \left(F_{k+2}+\delta^2\right) 
 \left(F_{k+1}+\delta^2\right) \left(F_{k+3}+\delta^2\right)
 +(-A^{-4})^{k+6}(\delta^2-1)F_{k+2}F_{k+1}F_{k+3}.
 \label{property:Kb}
 \end{eqnarray*}
 By Theorem~\ref{theorem:bracket}, it follows that
 \[
 A^{12}\Kb{D(k;k+4,k+3,k+5)}=\Kb{D(k+6;k+2,k+1,k+3)}.
 \]
 Since $w(D(k;k+4,k+3,k+5))=k+4$ and $w(D(k+6;k+2,k+1,k+3))=k+8$,
 we obtain $\displaystyle V_{P(k;k+4,k+3,k+5)}=V_{P(k+6;k+2,k+1,k+3)}$.
\end{proof}

\vspace{3mm}

\begin{proof*}[of Proposition~\ref{prop:Alexander}]
 For an alternating knot $K$,
 it is well-known that the genus of $K$ is equal  
 $\frac12\deg\Delta_K$ and a minimal genus Seifert surface
 of $K$ is obtained from its alternating diagram by applying Seifert
 algorithm.
 It follows that 
 \[
 \deg\Delta_{P(k;k+4,k+3,k+5)}=2k+8,\quad
 \deg\Delta_{P(k+6;k+2,k+1,k+3)}=2k+4.\makebox[0cm]{\hspace*{38mm}\proofbox}
 \]
\end{proof*}

\begin{acknowledgements}
 The first author would like to thank Professor Taizo Kanenobu for his
 suggestion.
 This research was motivated by his e-mail about Jones polynomial of
 pretzel links.
\end{acknowledgements}


\bibliography{refsEng}

\bibliographystyle{plain}

\affiliationone{
   Masao Hara\\
   Department of Mathematical Sciences\\
   School of Science\\
   Tokai University\\
   4-1-1 Kitakaname, Hiratsuka-shi\\ Kanagawa, 259-1292 \\
   Japan
   \email{masao@tokai-u.jp}}
\affiliationtwo{
   Makoto Yamamoto\\
   Department of Mathematics\\
   Faculty of Science and Engineering\\
   Chuo University\\
   1-13-27 Kasuga, Bunkyo-ku\\ Tokyo, 112-8551\\
   Japan
   \email{makotoy@math.chuo-u.ac.jp}}

\end{document}